\documentclass{article}
\usepackage[utf8]{inputenc}

\usepackage{amssymb}
\usepackage{amsthm}
\usepackage{amsmath,amscd}
\usepackage[mathscr]{euscript}
\usepackage[all]{xy}
\usepackage{lmodern}
\usepackage[T1]{fontenc}
\usepackage[textwidth=14cm,hcentering]{geometry}
\usepackage[colorlinks=true,linkcolor=red,citecolor=blue]{hyperref}
\usepackage{cleveref}
\usepackage{mathtools}
\usepackage{xspace}
\usepackage{tikz}
\usetikzlibrary{cd}
\usetikzlibrary{arrows,positioning}
\usepackage{bbm}
\usepackage{aliascnt}

\usepackage{etoolbox}
\AtBeginEnvironment{thebibliography}{\setlength{\itemsep}{0pt}}

\title{$\Mot$ is not compactly generated}

\author{Maxime Ramzi}
\date{}

\newtheorem{thm}{Theorem}[section]
\newaliascnt{lm}{thm}  

\aliascntresetthe{lm}
\Crefname{lm}{Lemma}{Lemmas}
\newaliascnt{prop}{thm}  
\newtheorem{prop}[prop]{Proposition}
\aliascntresetthe{prop}
\Crefname{prop}{Proposition}{Propositions}
\newaliascnt{cor}{thm}  
\newtheorem{cor}[cor]{Corollary}
\aliascntresetthe{cor}
\Crefname{cor}{Corollary}{Corollaries}

\newtheorem*{thm*}{Theorem}
\newtheorem*{cor*}{Corollary}

\theoremstyle{definition}
\newaliascnt{defn}{thm}  

\aliascntresetthe{defn}
\Crefname{defn}{Definition}{Definitions}
\newaliascnt{cons}{thm}  

\aliascntresetthe{cons}
\Crefname{cons}{Construction}{Constructions}
\newaliascnt{nota}{thm}  

\aliascntresetthe{nota}
\Crefname{nota}{Notation}{Notations}
\newaliascnt{conv}{thm}  

\aliascntresetthe{conv}
\Crefname{conv}{Convention}{Conventions}
\newaliascnt{ex}{thm}  

\aliascntresetthe{ex}
\Crefname{ex}{Example}{Examples}
\newaliascnt{rmk}{thm}  
\newtheorem{rmk}[rmk]{Remark}
\aliascntresetthe{rmk}
\Crefname{rmk}{Remark}{Remarks}
\newaliascnt{ques}{thm}  

\aliascntresetthe{ques}
\Crefname{ques}{Question}{Questions}
\newaliascnt{conj}{thm}  

\aliascntresetthe{conj}
\Crefname{conj}{Conjecture}{Conjectures}
\newaliascnt{warn}{thm}  
\newtheorem{warn}[warn]{Warning}
\aliascntresetthe{warn}
\Crefname{warn}{Warning}{Warnings}
\newaliascnt{obs}{thm}  

\aliascntresetthe{obs}
\Crefname{obs}{Observation}{Observations}
\newtheorem*{ques*}{Question}
\newtheorem*{rmk*}{Remark}
\newtheorem*{ex*}{Example}
\newaliascnt{recoll}{thm}  

\aliascntresetthe{recoll}
\Crefname{recoll}{Recollection}{Recollections}

\newcommand{\cat}{\mathrm}
\newcommand{\Cat}{\cat{Cat}}
\newcommand{\Catperf}{\mathrm{Cat}^{\mathrm{perf}}}
\newcommand{\on}{\operatorname}

\newcommand{\Fun}{\on{Fun}}

\newcommand{\Sph}{\mathbb S}

\newcommand{\Sp}{\cat{Sp}}

\newcommand{\PrL}{\mathrm{Pr}^\mathrm{L} }

\newcommand{\CAlg}{\mathrm{CAlg}}

\newcommand{\Mod}{\cat{Mod}}

\newcommand{\Mot}{\cat{Mot}^{\mathrm{loc}}}

\newcommand{\HH}{\mathrm{HH}}
\newcommand{\THH}{\mathrm{THH}}

\newcommand{\Ind}{\mathrm{Ind}}

\newcommand{\colim}{\mathrm{colim}}

\newcommand{\U}{\mathcal{U}_{\mathrm{loc}}}

\newcommand{\st}{\mathrm{st}}
\newcommand{\dbl}{\mathrm{dbl}}

\newcommand{\Q}{\mathbb{Q}}

\begin{document}

\maketitle
\begin{abstract}
    We prove that the $\infty$-category of localizing motives in the sense of Blumberg--Gepner--Tabuada is not compactly generated, extending a result of Efimov.  
\end{abstract}
\section*{Introduction}
The category\footnote{Here and throughout, we use ``categories'' to refer to $\infty$-categories.} of localizing motives, $\Mot$, introduced by Blumberg, Gepner and Tabuada \cite{BGT} has become a particularly important tool in the study of algebraic $K$-theory and related invariants. For a long time it has seemed hard to access, but recently Efimov has made spectacular progress in understanding its structural properties.

One key question about this category relates to its compact generation: while $\Catperf$ is compactly generated, the universal localizing invariant $\U:\Catperf\to \Mot$ does not preserve nor reflect compact objects -- this makes the understanding of compact objects in $\Mot$ somewhat subtle. In a breakthrough result, Efimov proved that it is \emph{dualizable}, a certain weakening of ``compactly generated'', earlier studied by Gaitsgory--Rozenblyum \cite{GR} and Efimov \cite{efimov}. This did not resolve the general question of compact generation, but Efimov \emph{did} prove that a relative version, $\Mot_{\Q[x]}$ is \emph{not} compactly generated \cite[Corollary 12.11]{rigidity}. This strongly suggests that the absolute version, $\Mot$, is also not compactly generated. In this paper, we prove that this expectation is correct:
\begin{thm*}\label{thm:main}
    Let $R$ be a connective $\mathbb E_2$-ring spectrum, and suppose that $R_\Q\neq 0$. The category $\Mot_R$ of motives relative to $R$ is not compactly generated. This is in particular the case for $R=\Sph$. 
    \end{thm*}
\begin{rmk*}
    As before, this suggests that for no nonzero $\mathbb E_2$-ring spectrum $R$ is $\Mot_R$ compactly generated. Our method does not immediately generalize to nonconnective ring spectra, nor to ring spectra that are supported on $(p)$ for some prime $p$, such as $\mathbb F_p$-algebras. 
\end{rmk*}
Our proof is essentially the observation that a result of Mathew's \cite{akhil} (generalizing earlier results of Kaledin \cite{kaledin,kaledin2}) already implies the non-compact generation. 
\section*{Conventions}
We use freely the theory of $\infty$-categories as developed in \cite{HTT,HA}, except that we use the term ``categories'' for ``$\infty$-categories''. We use standard notation such as $\Sp, \PrL_\st, \Fun^L$. We also use standard notation from classical references on localizing invariants \cite{BGT,efimov}, dualizable categories \cite{Dbl,efimov} and rigid categories \cite{locrig}.  We assume the reader is at least somewhat familiar with these. 
\section{Preliminaries}
First, we recall Efimov's rigidity theorem: 
\begin{thm}[{\cite[Theorem 0.3]{rigidity}}]\label{thm:rigidity}
Let $R$ be a commutative ring spectrum. The category $\Mot_R$ of $R$-linear motives is rigid in the sense of Gaitsgory--Rozenblyum \cite{GR}. In particular, its compact objects agree with its dualizable objects. 

Furthermore, it is dualizable, and in fact self-dual: the following functor is an equivalence: $$\Mot_R\to \Fun^L(\Mot_R,\Sp), M\mapsto K(M\otimes_R -)$$
\end{thm}
In fact, Efimov proves this in the generality of a rigid base $\mathcal E$: $\Mot_\mathcal E$ is also rigid. Let us also recall the following folklore result : 
\begin{prop}
    The functor $\Mot_{-}:\CAlg^{\mathrm{rig}}(\PrL_\st)\to \Pr^\dbl_\st$ preserves filtered colimits. 
\end{prop}
\begin{proof}
    Efimov gives a proof in \cite[Proposition 5.3]{rigidity}. We provide another one.  By \cite[Proposition 3.2.9]{stefanich} and \cite[Proposition 1.62]{Dbl} it suffices to prove the statement for the functor $\Mot_{-}$ with values in the category $\widehat{\Cat}^\omega$ of categories with filtered colimits and colimit-preserving functors. 

Consider a filtered diagram $C_\bullet: I\to \CAlg^{\mathrm{rig}}(\PrL_\st)$ with colimit $C$. We claim that $\colim_i \Mot_{C_i}$ and $\Mot_{C}$ have the same universal property in $\widehat{\Cat}^\omega$: by the main result of \cite{RSW} (see also Appendix B in \textit{loc. cit.}), the latter is a localization of $\Mod_{C}(\Pr^\dbl_\st)$ at $C$-linear motivic equivalences. Furthermore, by \cite[Proposition 3.2.9]{stefanich}, \cite[Proposition 1.62]{Dbl} and \cite[Corollary 4.8.5.13]{HA}, $\colim_i \Mod_{C_i}(\Pr^\dbl_\st)\simeq \Mod_{C}(\Pr^\dbl_\st)$ in $\widehat{\Cat}^\omega$.

Thus, it suffices to prove the following: if a filtered-colimit-preserving functor $f:\Mod_{C}(\Pr^\dbl_\st)\to D$ is such that for all $i$, the composite $ \Mod_{C_i}(\Pr^\dbl_\st)\to \Mod_{C}(\Pr^\dbl_\st)\to D$ inverts $C_i$-motivic equivalences, then $f$ inverts $C$-motivic equivalences. 

    This is clear: every motivic equivalence $A\to B$ can be written as the colimit $$\colim_i (C\otimes_{C_i}A\to C\otimes_{C_i}B)$$ which are base changed from $C_i$-motivic equivalences.
\end{proof}

The key corollary we will use and combine with Mathew's result is the following: 
\begin{cor}\label{cor:lift}
The functor $(\Mot_{-})^\dbl:\CAlg^{\mathrm{rig}}(\PrL_\st)\to \Catperf$ preserves filtered colimits. 
\end{cor}
\begin{proof}
This follows from rigidity, the previous proposition and \cite[Proposition 2.50]{Dbl}.
\end{proof}
\section{The proof}
We recall the following theorem of Mathew's: 
\begin{thm}\label{thm:Akhil}
    Let $M$ be a dualizable\footnote{Equivalently, compact.} motive over a connective rational $\mathbb E_\infty$-ring $A$. In this case, the $S^1$-action on $\HH(M/A)$ is trivial. 
\end{thm}
\begin{proof}
    This ``is'' \cite[Corollary 4.8]{akhil} and the second paragraph of Section 3 in \textit{loc. cit.}. Note that despite the notation, Mathew's $\mathcal N\mathrm{Mot}_A^\omega$ is not $(\Mot_A)^\omega$ in our notation: Mathew only considers motives of smooth and proper $A$-linear categories (which are in particular dualizable, but it is not expected that all dualizable motives are of this form). However, the only place in which Mathew uses that he is dealing with smooth and proper $A$-linear categories as opposed to dualizable $A$-linear motives is to show that they lift to a smooth and proper $R$-linear category where $R$ is a compact connective $\mathbb E_\infty$-algebra. By \Cref{cor:lift}, this is also the case for dualizable motives, and from there on the proof proceeds unchanged. 
\end{proof}
\begin{warn}
    Crucially, ``trivial'' really means ``trivial'', and not something like ``unipotent'' (i.e. in the thick subcategory generated by trivial actions). 
\end{warn}

We are now essentially done: 
\begin{cor}\label{cor:main}
    Let $M$ be a motive over a non-zero connective rational $\mathbb E_\infty$-ring $A$. Suppose $\HH(M/A)\in \Mod_A^{BS^1}$ contains a summand of the form $A[S^1]$. Then $M$ is not in $\Ind((\Mot_A)^\omega)$. In particular, if such an $M$ exists, $\Mot_A$ is not compactly generated.
\end{cor}
\begin{proof}
    Suppose $M\simeq \colim_i M_i$ with each $M_i$ compact. By \Cref{thm:rigidity}, each $M_i$ is dualizable and hence $\HH(M_i/A)$ has a trivial $S^1$-action by \Cref{thm:Akhil}. Since $A[S^1]$ is compact as an $A$-module with $S^1$-action, the composite $$A[S^1]\to \HH(M/A)\simeq \colim_i \HH(M_i/A)$$ factors through an $\HH(M_i/A)$ and hence the retraction $\HH(M/A)\to A[S^1]$ provides a retraction of a trivial $S^1$-action onto $A[S^1]$. This contradicts $A\neq 0$. 
\end{proof}
This is already sufficient to prove our main result in the case of an $\mathbb E_\infty$-ring. Since we stated the result for $\mathbb E_2$-rings, we have to add an extra reduction step.
\begin{proof}[Proof of the Theorem]
    Let $R$ be a connective $\mathbb E_2$-ring with $R_\Q\neq 0$. Note that $A=\pi_0(R)_\Q$ is a nonzero rational commutative ring, and hence commutative ring spectrum. Consider the canonical map $f:R\to A$. 

    Consider now the square zero $R$-algebra $S=R\oplus \Sigma R$. By \cite{hesselholt} (see also \cite[Proposition 4.5.1]{raskin}), $\HH(S/R)$ contains $R[S^1]$ as a summand. By \Cref{cor:main}, the $A$-linear motive of $A\otimes_R S$ is not in $\Ind((\Mot_A)^\omega)$. Since $\Mot_R\to \Mot_A$ is an internal left adjoint, it preserves $\Ind((\Mot_{-})^\omega)$, whence we deduce that the $R$-linear motive of $S$ is not in $\Ind((\Mot_R)^\omega)$, which proves the claim. 
\end{proof}
\begin{rmk}
    By considering instead the motive $\mathrm{thh} \in \Mot$, defined by the condition $K(\mathrm{thh}\otimes -)\simeq \THH$ (cf. \Cref{thm:rigidity}), we can obtain a motive $M$ such that $\THH(M)\simeq \Omega \Sph[S^1]$ on the nose (cf. \cite[Corollary 2.2]{EndTHH}), rather than having it as a retract. The argument is simpler with the square zero extension, though. 
\end{rmk}
\section*{Acknowledgements}
Thanks to Phil Pützstück for some comments on a draft.

This research was funded by the Deutsche Forschungsgemeinschaft (DFG, German Research Foundation) – Project-ID 427320536 – SFB 1442, as well as under Germany's Excellence Strategy EXC 2044/2 –390685587, Mathematics Münster: Dynamics–Geometry–Structure.
\newline 
\textbf{AI Disclosure: } No AI was used for this research.



\begin{thebibliography}{Lur12}

\bibitem[AMN18]{KunnethTP}
Antieau, B., Mathew, A., and Nikolaus, T.
\newblock On the Blumberg--Mandell K\"unneth theorem for {TP}.
\newblock \emph{Selecta Mathematica}, 24(5):4555--4576, 2018.

\bibitem[BGT13]{BGT}
Blumberg, A.~J., Gepner, D., and Tabuada, G.
\newblock A universal characterization of higher algebraic {K}-theory.
\newblock \emph{Geometry \& Topology}, 17(2):733--838, 2013.

\bibitem[Efi25a]{efimov}
Efimov, A.~I.
\newblock K-theory and localizing invariants of large categories.
\newblock \href{https://arxiv.org/abs/2405.12169}{arXiv:2405.12169}, 2025.


\bibitem[Efi25b]{rigidity}
Efimov, A.~I.
\newblock Rigidity of the category of localizing motives.
\newblock \href{https://arxiv.org/abs/2510.17010}{arXiv:2510.17010}, 2025.

\bibitem[GR19]{GR}
Gaitsgory, D. and Rozenblyum, N.
\newblock \emph{A study in derived algebraic geometry: Volume {I}: correspondences and duality}, volume~221.
\newblock American Mathematical Society, 2019.

\bibitem[Hes94]{hesselholt}
Hesselholt, L.
\newblock Stable topological cyclic homology is topological {Hochschild} homology.
\newblock \emph{Asterisque}, 226:175--192, 1994.

\bibitem[Kal08]{kaledin}
Kaledin, D.
\newblock Non-commutative {Hodge-to-de Rham} degeneration via the method of {Deligne-Illusie}.
\newblock \emph{Pure and Applied Mathematics Quarterly}, 4(3):785--876, 2008.

\bibitem[Kal16]{kaledin2}
Kaledin, D.
\newblock Spectral sequences for cyclic homology.
\newblock \href{https://arxiv.org/abs/1601.00637}{arXiv:1601.00637}, 2016.

\bibitem[Lur09]{HTT}
Lurie, J.
\newblock \emph{Higher topos theory}.
\newblock Princeton University Press, 2009.

\bibitem[Lur12]{HA}
Lurie, J.
\newblock \emph{Higher algebra}.
\newblock 2017. Available at author's webpage.


\bibitem[Mat20]{akhil}
Mathew, A.
\newblock Kaledin’s degeneration theorem and topological {Hochschild} homology.
\newblock \emph{Geometry \& Topology}, 24(6):2675--2708, 2020.

\bibitem[Ram24a]{Dbl}
Ramzi, M.
\newblock Dualizable presentable {$\infty$}-categories.
\newblock \href{https://arxiv.org/abs/2410.21537}{arXiv:2410.21537}, 2024.

 

\bibitem[Ram24b]{locrig}
Ramzi, M.
\newblock Locally rigid {$\infty$}-categories.
\newblock \href{https://arxiv.org/abs/2410.21524}{arXiv:2410.21524}, 2026.

\bibitem[Ram25a]{EndTHH}
Ramzi, M.
\newblock On endomorphisms of topological {Hochschild} homology.
\newblock \href{https://arxiv.org/abs/2503.04438}{arXiv:2503.04438}, 2025.

\bibitem[RSW25]{RSW}
Ramzi, M., Sosnilo, V., and Winges, C.
\newblock Every motive is the motive of a stable {$\infty$}-category.
\newblock \href{https://arxiv.org/abs/2503.11338}{arXiv:2503.11338}, 2025.

\bibitem[Ras18]{raskin}
Raskin, S.
\newblock On the {Dundas-Goodwillie-McCarthy} theorem.
\newblock \href{https://arxiv.org/abs/1807.06709}{arXiv:1807.06709}, 2018.


\bibitem[Ste25]{stefanich}
Stefanich, G.
\newblock Classification of fully dualizable linear categories.
\newblock \href{https://arxiv.org/abs/2307.16337}{arXiv:2307.16337}, 2025.


\end{thebibliography}

\end{document}